\documentclass[12pt]{article}
\usepackage{eucal}
\usepackage[utf8]{inputenc}
\usepackage[english]{babel}
\usepackage{amsfonts,amsmath,amssymb,amsthm}
\usepackage[all]{xy}
\newtheorem{trm}{Theorem}[section]
\newtheorem*{trm*}{Theorem}
\newtheorem{dfz}[trm]{Definition}
\newtheorem{prop}[trm]{Proposition}
\newtheorem{lmm}[trm]{Lemma}
\newtheorem{cor}[trm]{Corollary}
\newcommand{\rmk}{\noindent\textbf{Remark. }}
\newcommand{\akn}{\vspace{15pt}\noindent\textbf{Acknowledgments }}
\usepackage{mathrsfs}
\newcommand{\B}{\mathcal{B}}
\newcommand{\hu}[1]{\hat{H}^{-1}(G_{n,m},#1)}
\newcommand{\hU}[1]{\hat{H}^{1}(G_{n,m},#1)}
\newcommand{\hz}[1]{\hat{H}^0(G_{n,m},#1)}
\newcommand{\hq}[1]{\hat{H}^{q}(G_{n,m},#1)}
\newcommand{\huG}[1]{\hat{H}^{-1}(G,#1)}
\newcommand{\hUG}[1]{\hat{H}^{1}(G,#1)}
\newcommand{\hzG}[1]{\hat{H}^0(G,#1)}
\newcommand{\hqG}[1]{\hat{H}^{q}(G,#1)}
\numberwithin{equation}{section}
\begin{document}
\title{Cyclotomic Units and Class Groups in $\mathbb{Z}_p$-extensions of real abelian fields}
\author{Filippo A. E. Nuccio\footnote{Universit\`a ``La Sapienza'' - ROME, \texttt{nuccio@mat.uniroma1.it}}}
\date{}
\maketitle
\abstract{For a real abelian number field $F$ and for a prime $p$ we study the relation between the $p$-parts of the class groups and of the quotients of global units modulo cyclotomic units along the cyclotomic $\mathbb{Z}_p$-extension of $F$. Assuming Greenberg's conjecture about the vanishing of the $\lambda$-invariant of the extension, a map between these groups has been constructed by several authors, and shown to be an isomorphism if $p$ does not split in $F$. We focus in the split case, showing that there are, in general, non-trivial kernels and cokernels.}
\section{Introduction}
Let $F/\mathbb{Q}$ be a real abelian field of conductor $f$ and let $Cl_F$ be its ideal class group. A beautiful formula for the order of this class group comes from the group of Cyclotomic Units: this is a subgroup of the global units $\mathcal{O}_F^\times$ whose index is linked to the order of $Cl_F$. To be precise, we give the following definition (\cite{Sin80/81}, section 4):
\begin{dfz} For integers $n>1$ and $a$ not divisible by $n$, let $\zeta_n$ be a primitive $n$-th root of unity. Then $Norm^{\mathbb{Q}(\zeta_n)}_{F\cap \mathbb{Q}(\zeta_n)}(1-\zeta_n^a)\in F$ and we define the cyclotomic numbers $D_F$ to be the subgroup of $F^\times$ generated by $-1$ and $Norm^{\mathbb{Q}(\zeta_n)}_{F\cap \mathbb{Q}(\zeta_n)}(1-\zeta_n^a)$ for all $n>1$ and all $a$ not divisible by $n$. Then we define the Cyclotomic Units of $F$ to be
$$
Cyc_F:=D_F\cap \mathcal{O}_F^\times
$$
\end{dfz}

Sinnott proved in \cite{Sin80/81}, Theorem 4.1 together with Proposition 5.1, the following theorem:
\begin{trm*}[Sinnott] There exists an explicit constant $\kappa_F$ divisible only by $2$ and by primes dividing $[F:\mathbb{Q}]$ such that
$$
[\mathcal{O}_F^\times:Cyc_F]=\kappa_F|Cl_F|\;.
$$
\end{trm*}
Let now $p$ be an odd prime that does not divide $[F:\mathbb{Q}]$: by tensoring $\mathcal{O}_F^\times$, $Cyc_F$ and $Cl_F$ with $\mathbb{Z}_p$ we get an equality
$$
[\mathcal{O}_F^\times\otimes\mathbb{Z}_p:Cyc_F\otimes\mathbb{Z}_p]=|Cl_F\otimes\mathbb{Z}_p|
$$
and it is natural to ask for an algebraic interpretation of this. Moreover, observe that our assumption $p\nmid [F:\mathbb{Q}]$ makes the Galois group $\Delta:=\mathrm{Gal}(F/\mathbb{Q})$ act on the modules appearing above through one-dimensional characters, and we can decompose them accordingly: in the sequel we write $M(\chi)$ for every $\mathbb{Z}[\Delta]$-module $M$ to mean the submodule of $M\otimes \mathbb{Z}_p$ of $M$ on which $\Delta$ acts as $\chi$, where $\chi\in\hat{\Delta}$ (see the beginning of Section \ref{SecCycUnits} for a precise discussion). Then an even more optimistic question is to hope for a character-by-character version of Sinnott's theorem, namely
\begin{equation}\label{charbychar}
[\mathcal{O}_F^\times\otimes\mathbb{Z}_p(\chi):Cyc_F\otimes\mathbb{Z}_p(\chi)]\stackrel{?}{=}|Cl_F\otimes\mathbb{Z}_p(\chi)|\;
\end{equation}
and then ask for an algebraic interpretation of this. Although it is easy to see that these $\Delta$-modules are in general not isomorphic (see the example on page 143 of \cite{KraSch95}), it can be shown that they sit in an exact sequence for a wide class of fields arising in classical Iwasawa theory. More precisely, let $F_\infty/F$ be the cyclotomic $\mathbb{Z}_p$-extension of $F$ and let $\Gamma=\mathrm{Gal}(F_\infty/F)\cong\mathbb{Z}_p$: then
$$
F_\infty=\bigcup_{n\geq 0}F_n\supset\hdots\supset F_n\supset F_{n-1}\supset\hdots\supset F_0=F
$$
where $F_n/F$ is a cyclic extension of degree $p^n$ whose Galois group is isomorphic to $\Gamma/\Gamma^{p^n}$. In a celebrated work (see \cite{Iwa73}) Iwasawa gives a formula for the growth of the order of $Cl_{F_n}\otimes\mathbb{Z}_p$: he proves that there are three integers $\mu,\lambda$ and $\nu$, and an index $n_0\geq 0$, such that
$$
|Cl_{F_n}\otimes\mathbb{Z}_p|=p^{\mu p^n+\lambda n+\nu}\; \mathrm{\;for\;every\;}n\geq n_0\;.
$$
Moreover, Ferrero and Washington proved in \cite{FerWas79} that the invariant $\mu$ vanishes. A long-standing conjecture by Greenberg (see \cite{Gre76}, where conditions for this vanishing are studied) predicts that $\lambda=0$: according to the conjecture the $p$-part of the class groups should stay bounded in the tower.

Although a proof of this conjecture has never been provided, many computational checks have been performed verifying the conjecture in many cases (see, for instance, \cite{KraSch95}). Under the assumptions $\lambda=0$ and $\chi(p)\ne 1$, \emph{i. e.} $p$ does not split in $F$, some authors (see \cite{BelNQD01}, \cite{KraSch95}, \cite{Kuz96} and \cite{Oza97}) were able to construct an explicit isomorphism
\begin{equation}\label{millemila}
\alpha:\big(Cl_{F_n}\otimes\mathbb{Z}_p\big)(\chi)\cong\big(\mathcal{O}_{F_n}^\times/Cyc_{F_n}\otimes\mathbb{Z}_p\big)(\chi)
\end{equation}
if $n$ is big enough. Although the construction of the above morphism works also in the case $\chi(p)=1$, as detailed in the beginning of Section \ref{Main Result}, the split case seems to have never been addressed. We focus then on this case, and study the map in this contest, still calling it $\alpha$. Our main result is the following (see Corollary \ref{Renefruit})
\begin{trm*} With notations as above, assume that $\chi$ is a character of $\Delta$ such that $\chi(p)=1$ and that $\lambda=0$. Then, for sufficiently big $n$, there is an exact sequence	
$$
0\longrightarrow \mathrm{K} \rightarrow \big(Cl_{F_n}\otimes\mathbb{Z}_p\big)(\chi)\stackrel{\alpha}{\rightarrow}\big(\mathcal{O}_{F_n}^\times/Cyc_{F_n}\otimes\mathbb{Z}_p\big)(\chi)\rightarrow \mathrm{C}\rightarrow 0\;:
$$
both the kernel $\mathrm{K}$ and the cokernel $\mathrm{C}$ of $\alpha$ are cyclic groups with trivial $\Gamma$-action of order $|L_p(1,\chi)|_p^{-1}$ where $L_p(s,\chi)$ is the Kubota-Leopoldt $p$-adic $L$-function.
\end{trm*}

\akn This work is part of my PhD thesis, written under the supervision of René Schoof. I would like to take this opportunity to thank him not only for proposing me to work on this subject and for the help he gave me in writing this paper, but especially for all the time and patience he put in following me through my PhD and for the viewpoint on Mathematics he suggested me.
\section{Some Tate Cohomology}\label{SecTate}
In this section we briefly recall some well-known facts that are useful in the sequel. Throughout, $L/K$ is a cyclic extension of number fields, whose Galois group we denote by $G$. In our application, $K$ and $L$ will usually be layers $F_m$ and $F_n$ of the cyclotomic $\mathbb{Z}_p$-extension for some $n\geq m$, but we prefer here not to restrict to this special case.

We need to introduce some notation. Let
$$
\mathbb{U}_K=\prod_{v\nmid \infty}\mathcal{O}_{K,v}^\times\times\prod_{v\mid \infty}K_v^\times
$$
be the id\`ele units, \emph{i. e.} id\`eles having valuation $0$ at all finite place $v$ (we refer the reader to sections $14-19$ of Cassels' paper in \cite{CasFro86} for basic properties of id\`eles and id\`eles class group) and let $\Sigma$ be the set of places of $K$ that ramify in $L/K$. It is known (see section $1.4$ of Serre's paper in \cite{CasFro86}) that the Tate cohomology of local units in an unramified extension of local fields is trivial: therefore Shapiro's Lemma (see Proposition 7.2 of Tate's paper in \cite{CasFro86}, ) gives
$$
\hqG{\mathbb{U}_L}=\hqG{\prod_{v\in\Sigma}\prod_{w\mid v}\mathcal{O}_{L,w}^\times}\cong\prod_{v\in\Sigma}\hat{H}^q(G_v,\mathcal{O}_{L,w}^\times),
$$
where we fix a choice of a place $w$ of $L$ above $v$ for every $v\in\Sigma$ and we denote by $G_v$ its decomposition group in $G$; we will make this identification throughout. We denote the product of local units at places $v\in\Sigma$ appearing above by $U_\Sigma$. Consider the following commutative diagram of $G$-modules:
\begin{equation}\label{magic square}
\xymatrix{&0\ar@{>}[1,0]&0\ar@{>}[1,0]&0\ar@{>}[1,0]&\\
0\ar@{>}[0,1]&\mathcal{O}_L^\times\ar@{>}[1,0]\ar@{>}[0,1] &L^\times\ar@{>}[1,0]\ar@{>}[0,1]&\mathrm{Pr}_L\ar@{>}[1,0]\ar@{>}[0,1]&0\\
0\ar@{>}[0,1]&\mathbb{U}_L\ar@{>}[1,0]\ar@{>}[0,1]& \mathbb{A}_L^\times \ar@{>}[1,0]\ar@{>}[0,1]& \mathrm{Id}_L\ar@{>}[1,0]\ar@{>}[0,1]&0 \\
0\ar@{>}[0,1]&Q_L\ar@{>}[1,0]\ar@{>}[0,1] &C_L\ar@{>}[1,0]\ar@{>}[0,1]& Cl_L\ar@{>}[1,0]\ar@{>}[0,1]&0\\
&0&0&0}
\end{equation}
Here $\mathrm{Id}_L$ and $\mathrm{Pr}_L$ denote the group of all fractional ideals of $L$ and of principal ideals, respectively; while $C_L$ is the group of id\`ele classes and $Q_L=\mathbb{U}_L/\mathcal{O}_L^\times$.

\begin{lmm}\label{primi} Consider the following diagram induced by (\ref{magic square}):
\begin{equation*}\label{lemma primi}
\xymatrix{
&\huG{\mathbb{U}_L}\ar@{>}[1,0]^\beta\\
\hzG{Cl_L}\ar@{>}[0,1]^\pi&\huG{Q_L}\;.}
\end{equation*}
Then $\mathrm{Im}(\beta)=\pi(\overline{\Sigma^G})$ where $\Sigma^G$ are the primes in $L$ above $\Sigma$ fixed by $G$ and $\overline{\Sigma^G}$ is their image in $\hzG{Cl_L}$.
\end{lmm}
\begin{proof} First of all, the above decomposition
$$
\huG{U_\Sigma}\cong \prod_{v\in\Sigma}\huG{\prod_{w\mid v}\mathcal{O}_w^\times}
$$
allows to write $\beta=\prod\beta_v$ where $\beta_v$ is the restriction of $\beta$ to the $v$-th factor $\huG{\prod \mathcal{O}_w^\times}$; we therefore fix a place $v\in\Sigma$ and we show that $\mathrm{Im}(\beta_v)=\pi(\overline{\mathfrak{p}_{w_1}\cdots\mathfrak{p}_{w_g}})$ where the $\mathfrak{p}_{w_i}$'s are the primes of $L$ above $v$. This follows from the fact that $\huG{\prod \mathcal{O}_w^\times}$ is a product of $g$ (the number of places $w\mid v$) cyclic groups, each of order $e_v$, the ramification index of $v$ in $L/K$. Now fix a uniformizer $\pi_{w_i}\in\mathcal{O}_{w_i}$ for all $w_i\mid v$ and chose a generator $\sigma_v$ of $G_v$: we have 
$$
\beta_v(\pi_{w_1}^{1-\sigma_{v}},\hdots,\pi_{w_g}^{1-\sigma_{v}})=\pi(\overline{\mathfrak{p}_{w_1}\cdots\mathfrak{p}_{w_g}})\;,
$$
as can immediately be seen from the commutativity of
\begin{equation*}\xymatrix{
\hzG{\mathrm{Id}_L}\ar@{>}[0,1]\ar@{>}[1,0]^{1-\sigma}&\huG{\prod_{w\mid v}\mathcal{O}_w^\times}\ar@{>}[1,0]^{\beta_v}\\
\hzG{Cl_L}\ar@{>}[0,1]^\pi&\huG{Q_L}
}\end{equation*}
where $\sigma$ is a generator of $G$ (inducing $\sigma_v$ through $G_v\hookrightarrow G$).
\end{proof}
\begin{prop}[See \cite{Iwa73}]\label{ker-capit} Let $\jmath: Cl_K\rightarrow Cl_L$ be the map induced by extending fractional ideals of $\mathcal{O}_K$ to $\mathcal{O}_L$. Then
$$
\mathrm{Ker}(\jmath)\cong\mathrm{Ker}\Big(\hUG{\mathcal{O}_L^\times}\rightarrow\hUG{U_\Sigma}\Big)\;.
$$
\end{prop}
\begin{proof}
We simply apply Snake Lemma twice. First of all, apply it to
\begin{equation*}\xymatrix{
0\ar@{>}[0,1]&Q_K\ar@{>}[0,1]\ar@{>}[1,0]&C_K\ar@{>}[0,1]\ar@{=}[1,0]&Cl_K\ar@{>}[0,1]\ar@{>}[1,0]^\jmath&0\\
0\ar@{>}[0,1]&Q_L^G\ar@{>}[0,1]&C_L^G\ar@{>}[0,1]&Cl_L^G\ar@{>}[0,1]&\hUG{Q_L}\;:
}
\end{equation*}
it shows $\mathrm{Ker}(\jmath)\cong Q_L^G/Q_K$. Then apply it to
\begin{equation*}\xymatrix{
0\ar@{>}[0,1]&\mathcal{O}_K^\times\ar@{>}[0,1]\ar@{=}[1,0]&\mathbb{U}_K\ar@{>}[0,1]\ar@{=}[1,0]&Q_K\ar@{>}[0,1]\ar@{>}[1,0]&0\\
0\ar@{>}[0,1]&(\mathcal{O}_L^\times)^G\ar@{>}[0,1]&(\mathbb{U}_L)^G\ar@{>}[0,1]&Q_L^G\ar@{>}[0,1]&\hUG{\mathcal{O}_L^\times}\;,
}
\end{equation*}
finding $Q_L^G/Q_K\cong\mathrm{Ker}\Big(\hUG{\mathcal{O}_L^\times}\rightarrow\hUG{U_\Sigma}\Big)$.
\end{proof}
\rmk The above proof does not use the hypothesis that $G$ be cyclic. In fact, we will in the sequel we apply the proposition assuming this cyclicity, finding $\mathrm{Ker}(\jmath)\cong\mathrm{Ker}\Big(\huG{\mathcal{O}_L^\times}\rightarrow\huG{U_\Sigma}\Big)\;$. We remark that in his thesis \cite{Gre76} Greenberg gives the following criterion: $\lambda=0$ if and only if the map $\jmath$ relative to the cyclic extension $F_n/F$, restricted to $p$-Sylow subgroups, becomes $0$ for sufficiently large $n$. This will be the starting point for our proof of Theorem \ref{LHF}.\\

Finally, we state a Lemma about the cohomology of the direct product of two groups. The main source for this is \cite{Sch90}. Suppose that $F$ is a Galois subfield of $K$ such that $L/F$ is Galois with $\mathrm{Gal}(L/F)\cong \Delta\times G$ where $\Delta=\mathrm{Gal}(K/F)$ is abelian and the isomorphism above is induced by restriction. Then every ``arithmetic'' module attached to $L$ comes equipped with a natural action of $\Delta\times G$ and we want to compare this action with the natural one on the Tate cohomology of $G$. We have the following
\begin{lmm}\label{cohomRene'}[\cite{Sch90}, Section 4] Suppose that $G$ is a $p$-group and that $p\nmid |\Delta|$. Let $M$ be a $\mathbb{Z}[\Delta\times G]-module$: then, for every $q\in\mathbb{Z}$, the natural map
$$
\hqG{M\otimes\mathbb{Z}_p}^\Delta\longrightarrow\hat{H}^q\Big(G,(M\otimes\mathbb{Z}_p)^\Delta\Big)\;.
$$
is an isomorphism (of abelian groups with trivial $\Delta\times G$-action).
\end{lmm}
\section{Cyclotomic Units}\label{SecCycUnits}

Fix from now on a non-trivial character $\chi\neq 1$ of $\Delta$ and an odd prime $p\nmid [F:\mathbb{Q}]$ such that $\chi(p)=1$. We set $R=\mathbb{Z}_p[\mathrm{Im}(\chi)]$ and we let  $\delta\in \Delta$ act on $R$ by $x\mapsto\chi(\delta)x$. In this way, $R$ becomes a $\mathbb{Z}_p[\Delta]$-algebra. For every $\mathbb{Z}[\Delta]$-module $M$ we denote by $M(\chi)$ the $\chi$-eigenspace of $M\otimes\mathbb{Z}_p$ for the action of $\Delta$: there is an  isomorphism of $R$-modules $M(\chi^{-1})\cong\big((M\otimes\mathbb{Z}_p)\otimes_{\mathbb{Z}_p}R\big)^\Delta$ (our notation is consistent with that of \cite{Rub00}, Section $1.6$, where all this is properly spelled out: we denote by $R$ what Rubin calls $\mathcal{O}$). In particular, with notations and assumption as in Lemma \ref{cohomRene'},
\begin{equation}\label{chicohom}
\hqG{M(\chi)}\cong \hqG{M}(\chi)\qquad\forall\;q\in\mathbb{Z}\;.
\end{equation}
It is easy to see that $Cl_F(\chi)$ is isomorphic via the norm to the $p$-part of the class group of $F^{\mathrm{Ker}\chi}$, a field in which $p$ splits completely (see \cite{Sch90} for more details). Analogously, $\mathcal{O}_F^\times/Cyc_F(\chi)$ is isomorphic to $\big(\mathcal{O}_{F^{\mathrm{Ker}\chi}}^\times/Cyc_{F^{\mathrm{Ker}\chi}}\big)\otimes\mathbb{Z}_p$: replacing if necessary $F$ by $F^{\mathrm{Ker}\chi}$ we can therefore assume that $p$ splits completely in $F$, and we assume this throughout.

We now go back to the situation described in the introduction: let then $F_\infty/F$ be the $\mathbb{Z}_p$-extension of $F$ and denote by $\Gamma$ its Galois group. For every $n$ and for every prime ideal $\mathfrak{p}\subseteq\mathcal{O}_{F_n}$ dividing $p$, we let $\mathcal{O}_{F_n,\mathfrak{p}}^{\times}$ denote the local units  at $\mathfrak{p}$. Then we set
$$
U_n:=\big(\prod_{\mathfrak{p}\mid p}\mathcal{O}_{F_n,\mathfrak{p}}^{\times})(\chi)
$$
and analogously
$$
\mathcal{O}_n^\times:=\mathcal{O}_{F_n}^{\times}(\chi)\quad\mathrm{and}\quad Cyc_n:=Cyc_{F_n}(\chi)\;.
$$
Observe that since in a $\mathbb{Z}_p$-extension only primes above $p$ ramify, the set $\Sigma$ of Section \ref{SecTate} relative to $F_n/F$ is $\Sigma=\{\mathfrak{p}\subset \mathcal{O}_F,\;\mathfrak{p}|p\}$ for all $n$ and our notation is consistent with the one introduced there. We finally set
\begin{equation*}
B_n:=\mathcal{O}_n^\times/Cyc_n\;,\quad\B_n:=U_n/Cyc_n\quad\text{and}\quad A_n:=Cl_{F_n}(\chi)\;.
\end{equation*}
By Sinnot's theorem above, together with Theorem 5.3 in \cite{Sin80/81} that guarantees $p\nmid \kappa_{F_n}$ for all $n$, the groups $Cl_{F_n}\otimes\mathbb{Z}_p$ and $\big(\mathcal{O}_{F_n}^\times/Cyc_{F_n}\big)\otimes\mathbb{Z}_p$ have the same order. The character-by-character version of this result is much deeper: it is known as Gras' Conjecture, and is a consequence of the (now proven) Main Conjecture of Iwasawa Theory, as detailed in \cite{Gre77} or in \cite{BelNQD01}. It follows that $|A_n|=|B_n|$ for all $n\geq 0$.\\

\rmk The semi-local units considered by Gillard in his paper \cite{Gil79-1} are products over all $\mathfrak{p}$ above $p$ of local units that are $1\pmod{\mathfrak{p}}$. In our situation, all completions $F_\mathfrak{p}$ at primes $\mathfrak{p}\mid p$ are isomorphic to $\mathbb{Q}_p$, so the two definitions coincide and $U_n$ is a free $\mathbb{Z}_p$-module of rank $1$.\\
Moreover, since $p$ splits completely in $F/\mathbb{Q}$, all primes above $p$ totally ramify in $F_n/F$ and the subgroup of fractional ideals in $F_n$ having support above $p$ is isomorphic to $\mathbb{Z}[\Delta]$. After tensoring with $R$, we can consider the $\chi$-eigenspace, that is still cyclic (as an $R$-module, now) and so is its projection to $A_n$. Thus, it makes sense to speak of \emph{the subgroup of $A_n$ of primes above $p$}, a cyclic group that we denote by $\Pi_n$. Since it is contained in $A_n^{G_n}$, we can use (\ref{chicohom}) to restate Lemma \ref{primi} saying that (with the same notation introduced there) $\beta(\hu{U_n})=\pi(\Pi_n)$ for all $n\geq m$.\\

We now investigate in some detail the structure of $Cyc_n$. First of all, letting $f$ be the conductor of $F$, we define the unit
$$
\eta_n:=\bigg(\mathrm{Norm}^{\mathbb{Q}(\zeta_{fp^{n+1}})}_{F_n}\big(1-\zeta_{fp^{n+1}}\big)\bigg)^{\sum_{\delta\in\Delta}\chi(\delta^{-1})\delta}\in Cyc_n\;.
$$
It is a unit since $p\nmid f$ because $p$ splits completely in $F$, it is cyclotomic by definition and we projected it in the $\chi$-component: morever, Sinnott's description of Cyclotomic Units shows that $Cyc_n=\eta_0\mathbb{Z}_p\times \eta_n\mathbb{Z}_p[G_n]$ (see, for instance, section 3 of \cite{Gil79-2} for details). In particular, we have an isomorphism of $G_n$-modules $Cyc_n\cong \mathbb{Z}_p\times I_{G_n}$ where $I_{G_n}$ is the augmentation ideal in $\mathbb{Z}_p[G_n]$, and we find a split exact sequence
\begin{equation}\label{?}
0\longrightarrow\langle\eta_n\rangle\longrightarrow Cyc_n\longrightarrow \langle\eta_0\rangle\longrightarrow 0
\end{equation}
where we denote, here and in what follows, by $\langle\eta_0\rangle$ and $\langle\eta_n\rangle$ the $\mathbb{Z}_p[G_n]$-modules generated by $\eta_0$ and $\eta_n$ respectively: by the above isomorphisms, (\ref{?}) corresponds to the sequence
\begin{equation}\label{??}
0\longrightarrow I_{G_n}\longrightarrow I_{G_n}\times \mathbb{Z}_p\longrightarrow \mathbb{Z}_p\longrightarrow 0\;.
\end{equation}
\begin{lmm}\label{cyc} We have $\hz{\langle\eta_n\rangle}=0$ and $\hu{\langle\eta_0\rangle}=0$. In particular, the natural map $\hq{Cyc_n}\cong\hq{\langle\eta_0\rangle}\times\hq{\langle\eta_n\rangle}$ induced by \ref{?} gives isomorphisms of $\Delta$-modules
$$\hz{Cyc_n}\cong\hz{\langle\eta_0\rangle}$$
and
$$
\hu{Cyc_n}\cong\hu{\langle\eta_n\rangle}\;.
$$
Both are cyclic groups of order $p^n$\;.
\end{lmm}
\begin{proof} The exact sequence
$$
0\longrightarrow I_{G_n}\longrightarrow \mathbb{Z}_p[G_n]\longrightarrow \mathbb{Z}_p\longrightarrow 0
$$
shows, since $\hq{\mathbb{Z}_p[G_n]}=0$ for all $q$, that $\hz{I_{G_n}}\cong\hu{\mathbb{Z}_p}=0$. The $\mathbb{Z}_p[G_n]$-isomorphisms $\langle\eta_n\rangle\cong I_{G_n}$ and $\langle\eta_0\rangle\cong\mathbb{Z}_p$ give the result.
\end{proof}
Suppose now that $\lambda=0$. The estension $F_\infty/F$ is totally ramified since $p$ splits in $F$ and the norm maps $N^n_m:A_n\rightarrow A_m$ are surjective by class field theory for all $n\geq m$; assuming $\lambda=0$ and chosing $m$ big enough, the orders of $A_n$ and $A_m$ coincide, and these norm maps are actually isomorphisms. Therefore the projective limit $X=\varprojlim A_n$ with respect to norms stabilizes to a finite group and $A_n\cong X$ for all $n\gg 0$ (we introduce here the notation $a\gg b$, equivalent to $b\ll a$, to mean that there exists a $b_0\geq b$ such that what we are stating holds for all $a\geq b_0$). In particular, the action of $\Gamma$ on $X$ must factor through a certain quotient $G_m=\Gamma/\Gamma^{p^m}$. Therefore $G_{n,m}$ acts trivially on $A_n$ for all $n\geq m$ and the $G_{n,m}$-norm $N_{G_{n,m}}=\sum_{\tau\in G_{n,m}}\tau$ acts on $A_n$ as multiplication by $p^{n-m}$. Choosing $n$ big enough so that $p^{n-m}A_n=0$, we find $N_{G_{n,m}}A_n=0$ and
\begin{equation}\label{n>>m}\begin{split}
\hz{A_n}=&A_n^{G_{n,m}}/N_{G_{n,m}}A_n=A_n\\
&=A_n[N_{G_{n,m}}]/I_{G_{n,m}}A_n=\hu{A_n}
\end{split}\end{equation}
where $I_{G_{n,m}}$ is the augmentation ideal of $\mathbb{Z}_p[G_{n,m}]$. Therefore $A_n\cong \hu{A_n}\cong\hz{A_n}$, whenever $\lambda=0$ and $n\gg m\gg 0$. A similar argument leads to the equivalent of (\ref{n>>m}) for $B_n$, namely $\hq{B_n}\cong B_n$ for all $q\in\mathbb{Z}$.
\begin{lmm}\label{crocerossa} If $\lambda=0$ and $m\gg 0$, the natural map
$$
H^{1}(G_{n,m},Cyc_n)\longrightarrow H^{1}(G_{n,m},\mathcal{O}_n^\times)
$$
is injective for all $n\geq m$.
\end{lmm}
\begin{proof} Taking $G_{n,m}$-cohomology in the exact sequence defining $B_n$ gives
\begin{multline}
0\longrightarrow H^0(G_{n,m},Cyc_n)\longrightarrow H^0(G_{n,m},\mathcal{O}_n^\times)\longrightarrow H^0(G_{n,m},B_n)\longrightarrow \\
\longrightarrow H^1(G_{n,m},Cyc_n)\longrightarrow H^1(G_{n,m},\mathcal{O}_n^\times)\;.
\end{multline}
Since $G_{n,m}$-invariants of $Cyc_n$ and $\mathcal{O}_n^\times$ are $Cyc_m$ and $\mathcal{O}_m^\times$ respectively, we find $\mathrm{Ker}\Big(H^{1}(G_{n,m},Cyc_n)\longrightarrow H^{1}(G_{n,m},\mathcal{O}_n^\times)\Big)=B_n^{G_{n,m}}/B_m$. Assuming that $m$ is big enough and $\lambda=0$ implies that the orders of $B_n$ and $B_m$ coincide, and the same holds, \emph{a fortiori}, for $B_n^{G_{n,m}}$ and $B_m$. Thus, the above kernel is trivial.
\end{proof}
\section{Semi-local Units modulo Cyclotomic Units}
We now state a very useful result about semi-local units in our setting; it can already be found in a paper by Iwasawa \cite{Iwa60}. We keep the same notation introduced in the previous section and we make from now on constant use of Lemma \ref{cohomRene'} above, especially in the form of the isomorphism in (\ref{chicohom}). 
\begin{dfz}\label{U1 e B1} We define $U_n^1$ to be the kernel $U_n^1=\mathrm{Ker}(N^n_0:U_n\rightarrow U_0)$ and we set $\B_n^1=U_n^1/\langle\eta_n\rangle$.
\end{dfz}
\begin{prop}\label{splitting} The natural map $U_n^1\times U_0\hookrightarrow U_n$ induced by injections is an isomorphism of $G_n$-modules. It induces a decomposition $\B_n\cong \B_n^1\times\B_0$.
\end{prop}
\begin{proof} Consider the exact sequence induced by the norm map $N^n_0$
$$
0\longrightarrow U_n^1\longrightarrow U_n\longrightarrow N^n_0(U_n)\subseteq U_0\longrightarrow 0\;.
$$
Since the extension $F_{n,\mathfrak{p}_n}/F_{0,\mathfrak{p}_0}$ is cyclic and totally ramified, local class field theory shows that $\hat{H}^0(G_n,U_n)=U_0/N^n_0(U_n)$ is a cyclic group of order $p^n$. As $U_0\cong\mathbb{Z}_p$ this shows $N^n_0(U_n)=U_0^{p^n}$ and since $U_0$ contains no roots of unity of $p$-power order we can identify this group with $U_0$ simply by extracting $p^n$-th roots. We find
\begin{equation}\label{splitseq}
0\longrightarrow U_n^1\longrightarrow U_n\stackrel{\sqrt[p^n]{N^n_0}}{\longrightarrow} U_0\longrightarrow 0\;.
\end{equation}
Since the natural embedding $U_0\hookrightarrow U_n$ is a $G_n$-linear section of (\ref{splitseq}), it splits the sequence and therefore gives an isomorphism $U_n^1\times U_0\cong U_n$.

The fact that this splitting induces an isomorphism $\B_n\cong\B_n^1\times \B_0$ simply follows from the commutativity of the following diagram:
$$\xymatrix{
0\ar@{>}[0,1]&\langle\eta_n\rangle\ar@{>}[0,1]\ar@{^(->}[1,0]&Cyc_n\ar@{>}[0,1]^{\sqrt[p^n]{N^n_0}}\ar@{^(->}[1,0]&\langle\eta_0\rangle\ar@{>}[0,1]\ar@{^(->}[1,0]&0\\
0\ar@{>}[0,1]&U_n^1 \ar@{>}[0,1]&U_n\ar@{>}[0,1]^{\sqrt[p^n]{N^n_0}}&U_0\ar@{>}[0,1]&0\;.
}$$
\end{proof}
More useful than the splitting itself is the following easy consequence:
\begin{cor}\label{cohom-splitting} For every $n\geq m\geq 0$ and for every $q\in\mathbb{Z}$ the natural maps
$$
\hq{U_n}\rightarrow\hq{U_n^1}\times\hq{U_0}
$$
and
$$
\hq{\B_n}\rightarrow\hq{\B_n^1}\times\hq{\B_0}\;.
$$
induced by Proposition \ref{splitting} are isomorphisms of abelian groups. In particular, we have identifications $\hu{U_n}=\hu{U_n^1}$ and $\hz{U_n}=\hz{U_0}$.
\end{cor}
\begin{proof} The splitting of the cohomology groups follows immediately from the Proposition. Concerning the cohomology of $U_n$, we observe that, since $G_{n,m}$ acts trivially on the torsion-free module $U_0$, the group $\hz{U_0}$ is cyclic of order $p^{n-m}$ while $\hu{U_0}=0$. This already implies that $\hu{U_n}= \hu{U_n^1}$. It also shows that $\hz{U_n^1}$ must be trivial because $\hz{U_n}$ is itself cyclic of order $p^{n-m}$ by local class field theory.
\end{proof}

\begin{lmm}\label{ordine cohom} For every $m\geq 0$ and for every $n\gg m$ there are isomorphisms $\hq{\B_n^1}\cong \mathbb{Z}_p/L_p(1,\chi)$ and $\hq{\B_0}\cong \mathbb{Z}_p/L_p(1,\chi)$ holding for every $q\in\mathbb{Z}$.
\end{lmm}
\begin{proof} Let $\Lambda:=\mathbb{Z}_p[[T]]$ and fix an isomorphism $\varpi:\Lambda\cong\mathbb{Z}_p[[\Gamma]]$. This isomorphism fixes a choice of a topological generator $\varpi(1+T)=:\gamma_0$ of $\Gamma$ and we denote by $\kappa\in\mathbb{Z}_p^\times$ the element $\varepsilon_{cyc}(\gamma_0)$ where
$$
\varepsilon_{cyc}:\mathrm{Gal}(\bar{F}/F)\longrightarrow \mathbb{Z}_p^\times
$$
is the cyclotomic character of $F$. The main tool of the proof will be Theorem 2 of \cite{Gil79-1}, that gives isomorphisms of $\mathbb{Z}_p[[\Gamma]]$-modules
\begin{equation}\label{gil=}
\B_0\cong \mathbb{Z}_p/L_p(1,\chi)\qquad\text{and}\qquad\B_n\cong \Lambda/(f(T),\omega_n(T)/T)
\end{equation}
where $\omega_n(T)=(1+T)^{p^n}-1$ and $f(T)\in \Lambda$ is the power series verifying $f(\kappa^s-1)=L_p(1-s,\chi)$ for all $s\in\mathbb{Z}_p$. We make $\Gamma$ act on the modules appearing in (\ref{gil=}) by $\gamma_0\cdot x=\varpi(\gamma_0)x=(1+T)x$ for all $x\in \B_0$ (\emph{resp.} all $x\in \B_n$): this induces the action of $G_{n,m}$ we need to compute the cohomology with respect to.

Starting with $\B_0$, observe that the action of $\Gamma$, and thus of its subquotient $G_{n,m}$, is trivial on the \emph{finite group} $\B_0$: as in (\ref{n>>m}) we get
$$
\hq{\B_0}\cong\B_0\qquad\text{for all }n\gg m\gg 0\;.
$$
and we apply (\ref{gil=}) to get our claim.

Now we compute $\hu{\B_n^1}$: by definition, $\hu{\B_n^1}=\B_n^1[N_{G_{n,m}}]/I_{G_{n,m}}\B_n^1$. Applying $\varpi$ we find $I_{G_{n,m}}\cong \omega_m(T)(\Lambda/\omega_n(T))$ and $\varpi(N_{G_{n,m}})=\nu_{n,m}(T)$ where $\nu_{n,m}(T):=\omega_n(T)/\omega_m(T)$. Hence
$$
\hu{\B_n^1}\cong\frac{\{g(T)\in\Lambda\;\Big\vert\; g(T)\nu_{n,m}(T)\in (f(T),\omega_n(T)/T)\}}{(f(T),\omega_n(T)/T,\omega_m(T))}\;.
$$
As observed in \cite{Gil79-1}, Lemma 5, $f(T)$ and $\omega_n(T)/T$ have no common zeroes. Therefore a relation
$$
g(T)\frac{\omega_n(T)}{\omega_m(T)}=a(T)f(T)+b(T)\frac{\omega_n(T)}{T}
$$
implies $\nu_{n,m}(T)\mid a(T)$ and we find $g(T)=c(T)f(T)+b(T)\omega_m(T)/T$ for some $c(T)\in\Lambda$: thus,
\begin{equation*}\begin{split}
\hu{\B_n^1}&\cong\frac{(f(T),\omega_m(T)/T)}{(f(T),\omega_n(T)/T,\omega_m(T))}\\
&\cong \frac{(\omega_m(T)/T)}{(f(T),\omega_n(T)/T,\omega_m(T))}\;.
\end{split}\end{equation*}
The evaluation map $g(T)\omega_m(T)/T\mapsto g(0)$ gives an isomorphism
$$
\frac{(\omega_m(T)/T)}{(\omega_m(T))}\cong\mathbb{Z}_p
$$
and we find
$$
\frac{(\omega_m(T)/T)}{(f(T),\omega_n(T)/T,\omega_m(T))}\cong \mathbb{Z}_p/\big(f(0),\omega_n(0)\big)\;.
$$
Since this last module is $\mathbb{Z}_p/f(0)$ as soon as $n$ is big enough and, by definition, $f(0)=L_p(1,\chi)$, we get our claim for $q=-1$. Using now that $\B_n^1$ is finite and therefore has a trivial Herbrand quotient, we know that the order of $\hz{\B_n^1}$ is again $|L_p(1,\chi)|_p^{-1}$: the fact that it is a cyclic group comes from the exact sequence
\begin{equation*}\begin{split}
0\rightarrow \hz{\B_n^1}\rightarrow\hz{U_n^1}\rightarrow&\hz{\langle\eta_n\rangle}\rightarrow\\
&\rightarrow\hU{\B_n^1}\rightarrow 0
\end{split}\end{equation*}
since $\hz{U_n^1}$ is itself cyclic, as discussed in Corollary \ref{cohom-splitting}.\\
Finally, the fact that $G_{n,m}$ is cyclic gives isomorphisms in Tate cohomology $\hat{H}^{2q}(G_{n,m},M)\cong \hat{H}^{0}(G_{n,m},M)$ for all modules $M$ (and analogously $\hat{H}^{2q+1}(G_{n,m},M)\cong \hat{H}^{-1}(G_{n,m},M)$), so the claim for  all $q$'s follows from our computation in the cases $q=0,-1$.
\end{proof}
\begin{prop} \label{ordine fissi} Recall that $X=\varprojlim A_n$: then, $|X^\Gamma|\stackrel{p}{=}L_p(1,\chi)$, where by $a\stackrel{p}{=}b$ we mean $ab^{-1}\in\mathbb{Z}_p^\times$.
\end{prop}
\begin{proof} Let $L_0$ be the maximal pro-$p$ abelian extension of $F_\infty$ everywhere unramified and let $M_0$ be the maximal pro-$p$ abelian extension of $F_\infty$ unramified outside $p$. We claim that $L_0=M_0$. This follows from the fact that for every $\mathfrak{p}\subseteq \mathcal{O}_F$ dividing $p$, the local field $F_\mathfrak{p}$ is $\mathbb{Q}_p$, since $p$ splits completely, and it therefore admits only two independent $\mathbb{Z}_p$-extensions by local class field theory. In particular, every pro-$p$ extension of $F_{\infty,\mathfrak{p}}$ that is abelian over $F_\mathfrak{p}$ must be unramified, so $M_0=L_0$. Now let $Y:=\mathrm{Gal}(L_\infty/F_\infty)$ where $L_\infty$ is the maximal pro-$p$ abelian extension of $F_\infty$ everywhere unramified: then the Artin reciprocity map gives an isomorphism $X\cong Y(\chi)$; also, let $M_\infty$ be the maximal pro-$p$ abelian extension of $F_\infty$ unramified outside $p$ and $\mathscr{Y}:=\mathrm{Gal}(M_\infty/F_\infty)$. A classical argument (see \cite{Was97}, chapter 13) shows that $Y_\Gamma=\mathrm{Gal}(L_0/F_\infty)$ and $\mathscr{Y}_\Gamma=\mathrm{Gal}(M_0/F_\infty)$: our claim above implies that $Y_\Gamma=\mathscr{Y}_\Gamma$. Since the actions of $\Delta$ and $\Gamma$ commute with each other, this also shows $X_\Gamma=\mathscr{Y}(\chi)_\Gamma$. Combine this with the following exact sequence induced by multiplication by $\gamma_0-1$ where $\gamma_0$ is a topological generator of  $\Gamma$
$$
0\longrightarrow X(\chi)^\Gamma\longrightarrow X(\chi)\stackrel{\gamma_0-1}{\longrightarrow}X(\chi)\longrightarrow X(\chi)_\Gamma\longrightarrow 0\;:
$$
it gives $|X^\Gamma|=|X_\Gamma|=|\mathscr{Y}(\chi)_\Gamma|$. The Main Conjecture of Iwasawa Theory, as proved by Rubin in the appendix of \cite{Lan90}, shows that the characteristic polynomial of $\mathscr{Y}(\chi)$ is $F(T)$ where $F(T)$ is the distinguished polynomial determined by $L_p(1-s,\chi)\stackrel{p}{=}F\big((1+p)^s-1\big)$ for all $s\in\mathbb{Z}_p$. Since $\mathscr{Y}$ contains no non-zero finite $\Gamma$-submodules (see \cite{NeuSchWin00}), we find $\mathscr{Y}^\Gamma=0$ and the order of $\mathscr{Y}(\chi)_\Gamma$ is $F(0)\stackrel{p}{=}L_p(1,\chi)$.
\end{proof}
\begin{cor}\label{ordine primi} If $\lambda=0$, then $\Pi_n$ is a cyclic group of order $|L_p(1,\chi)|_p^{-1}$ for every $n\gg 0$.
\end{cor}
\begin{proof} Indeed, Theorem 2 in \cite{Gre76} shows that $\lambda=0$ if and only if $X^\Gamma=\Pi_n$. The result now follows from the proposition and from the remark of section \ref{SecCycUnits}.
\end{proof}
\section{Main result}\label{Main Result}
We are now in position of proving our main result. We stick to the notation introduced in Section \ref{SecCycUnits}. Let $n\geq 0$ and let $Q_n:=(\mathbb{U}_{F_n})(\chi)/\mathcal{O}_n^\times=Q_{F_n}(\chi)$ as in Section \ref{SecTate}: consider the exact sequence
$$
0\longrightarrow \mathcal{O}_n^\times\longrightarrow \mathbb{U}_{F_n}(\chi)\longrightarrow Q_n\longrightarrow 0\;.
$$
Since $Cyc_n\subseteq \mathcal{O}_n^\times$, it induces an exact sequence
$$
0\longrightarrow B_n\longrightarrow  \mathbb{U}_{F_n}(\chi)/Cyc_n\longrightarrow Q_n\longrightarrow 0\;,
$$
and the Tate cohomology of $\mathbb{U}_{F_n}(\chi)/Cyc_n$ coincides with that of $\B_n$, as discussed in Section \ref{SecTate}. For every $m\leq n$ the cyclicity of Tate cohomology for cyclic groups induces an exact square
\begin{equation}\label{hex}\xymatrix{
\hz{Q_n}\ar@{>}[0,1]^{\alpha_{[0,1]}}&\hu{B_n}\ar@{>}[1,0]\\
\hz{\B_n}\ar@{>}[-1,0]&\hu{\B_n}\ar@{>}[1,0]\\
\hz{B_n}\ar@{>}[-1,0]&\hu{Q_n}\ar@{>}[0,-1]^{\alpha_{[1,0]}}\;.
}\end{equation}
Pick now $q\in\mathbb{Z}$ and consider the exact sequence
\begin{equation}\label{Q=A}
0\longrightarrow Q_n\longrightarrow C_{F_n}(\chi)\longrightarrow A_n\longrightarrow 0\;:
\end{equation}
as the actions of $G_{n,m}$ and of $\Delta$ commute, we have $\hq{C_n(\chi)}\cong\hq{C_n}(\chi)$; global class field theory (see section $11.3$ of Tate's paper in \cite{CasFro86}) shows that $\hq{C_n}(\chi)\cong\hat{H}^{q+2}(G_{n,m},\mathbb{Z}(\chi))=0$ because we assumed $\chi\neq 1$. Therefore the long exact cohomology sequence of (\ref{Q=A}) induces isomorphisms
\begin{equation}\label{Q&A}
\hq{A_n}\cong \hat{H}^{q+1}(G_{n,m},Q_n)\quad\mathrm{\;for\;every\;} q\in\mathbb{Z}\;.
\end{equation}
\rmk Observe that our discussion never uses the assumption $\chi(p)=1$. Indeed, the maps $\alpha_{[0,1]}$ and $\alpha_{[1,0]}$ are defined whenever $\chi\neq 1$ and are indeed the same maps appearing in Proposition 2.6 of \cite{KraSch95}, where the case $\chi(p)\ne 1$ is treated. As discussed in the introduction, in that case they turned out to be isomorphism if $\lambda=0$ (see also \cite{BelNQD01}, \cite{Kuz96} and \cite{Oza97}). We are going to see this is not the case if $\chi(p)=1$.
\begin{trm} \label{LHF} Assume $\lambda=0$ and $n\gg m\gg 0$. Then the kernels $\mathrm{Ker}(\alpha_{[0,1]})$, $\mathrm{Ker}(\alpha_{[1,0]})$ and the cokernels $\mathrm{Coker}(\alpha_{[0,1]})$, $\mathrm{Coker}(\alpha_{[1,0]})$ are cyclic groups of order $|L_p(1,\chi)|_p^{-1}$.
\end{trm}
\begin{proof} We start by determining $\mathrm{Ker}(\alpha_{[0,1]})$. Choose $m$ big enough so that $|A_{m+k}|=|A_m|$ for all $k\geq 0$ and $n\geq m$ big enough so that $\hq{A_n}=A_n$. As in the remark of Section \ref{SecTate}, Proposition 2 of \cite{Gre76} shows that $\lambda=0$ implies $A_m=\mathrm{Ker}(\jmath_{m,n})$ if $n$ is sufficiently large. Combining Proposition \ref{ker-capit} with (\ref{Q&A}), this gives an injection
\begin{equation}\label{injQ}
\hz{Q_n}\hookrightarrow \hu{\mathcal{O}_n^\times}\;.
\end{equation}
Consider now the following commutative diagram, whose row and column are exact and where the injectivity of the vertical arrow in the middle follows from Lemma \ref{crocerossa}:
\begin{equation}\label{ker}\xymatrix{
&\hu{B_n}\\
\hz{Q_n}\ar@{^(->}[0,1]\ar@{>}[-1,1]^{\alpha_{[0,1]}}&\hu{\mathcal{O}_n^\times}\ar@{>}[-1,0]\ar@{>}[0,1]&\hu{U_n}\\
&\hu{Cyc_n}\ar@{^(->}[-1,0]\ar@{>}[-1,1]^\psi\\
}\end{equation}
An easy diagram chase shows that $\mathrm{Ker}(\alpha_{[0,1]})\cong\mathrm{Ker}(\psi)$. In order to study $\mathrm{Ker}(\psi)$, observe that $\psi$ appears in the sequence
\begin{equation}\label{01}
0\to \hz{\B_n^1}\to\hu{\langle\eta_n\rangle}\stackrel{\psi}{\to}\hu{U_n^1}\;,
\end{equation}
because $\hu{Cyc_n}=\hu{\langle\eta_n\rangle}$ by Lemma \ref{cyc}, while $\hu{U_n}=\hu{U^1_n}$ by Corollary \ref{cohom-splitting}. Morever, again by Corollary \ref{cohom-splitting}, $\hz{U_n^1}=0$ and (\ref{01}) is exact, thus giving
\begin{equation}\label{ker01}
\mathrm{Ker}(\alpha_{[0,1]})\cong\hz{\B_n^1}\cong \mathbb{Z}_p/L_p(1,\chi)\;,
\end{equation}
the last isomorphisms being Lemma \ref{ordine cohom}.\\
Having determined $\mathrm{Ker}(\alpha_{[0,1]})$, the exactness of (\ref{hex}) together with Corollary \ref{cohom-splitting} (and Lemma \ref{ordine cohom}) show immediately that
\begin{equation}\label{cok10}
\mathrm{Coker}(\alpha_{[1,0]})\cong\hz{\B_0}\cong\mathbb{Z}_p/L_p(1,\chi)\;.
\end{equation}
Since the orders of $A_n$ and $B_n$ coincide for $n\gg 0$, and since these groups are isomorphic to $\hq{Q_n}$ and $\hq{B_n}$ respectively (see (\ref{n>>m}) and (\ref{Q&A})), the equalities of orders
$$
|\mathrm{Ker}(\alpha_{[0,1]})|=|\mathrm{Coker}(\alpha_{[0,1]})|\quad\text{and}\quad|\mathrm{Ker}(\alpha_{[1,0]})|=|\mathrm{Coker}(\alpha_{[1,0]})|
$$
hold. By (\ref{ker01}) and (\ref{cok10}), the four groups have the same order, equal to $|L_p(1,\chi)|_p^{-1}$.\\
We are left with the structure of $\mathrm{Ker}(\alpha_{[1,0]})$ and $\mathrm{Coker}(\alpha_{[0,1]})$. The map $\alpha_{[1,0]}$ is the composition
\begin{equation}\label{inclusione?}\xymatrix{
\hu{Q_n}\ar@{>}[0,1]^{\tilde{\beta}}\ar@/_2pc/[rr]_{\alpha_{[1,0]}}&\hz{\mathcal{O}_n^\times}\ar@{>}[0,1]&\hz{B_n}
}\end{equation}
and $\mathrm{Ker}(\alpha_{[1,0]})\supseteq\mathrm{Ker}(\tilde{\beta})=\Pi_n$, the last identification coming from Lemma \ref{primi}. Combining Corollary \ref{ordine primi} with the computation of the order of $\mathrm{Ker}(\alpha_{[1,0]})$ performed above, the inclusion cannot be strict, and $\mathrm{Ker}(\alpha_{[1,0]})$ is cyclic of the prescribed order. Looking at $\mathrm{Ker}(\alpha_{[1,0]})$ and at $\mathrm{Coker}(\alpha_{[0,1]})$ as subgroups of $\hu{\B_n}$ as in (\ref{hex}), and knowing the structure of this last module by Lemma \ref{cohom-splitting}, shows that $\mathrm{Coker}(\alpha_{[0,1]})$ is cyclic, too.
\end{proof}
Now we can single out from the proof a precise description of the kernels of the maps $\alpha_{[1,0]}$ and $\alpha_{[0,1]}$ when seen as maps
$$
\alpha_{[i,j]}:A_n\rightarrow B_n
$$
by combining (\ref{n>>m}) and (\ref{Q&A}). Before stating the next result, observe that, by Lemma (\ref{cyc}), $\hu{Cyc_n}\cong \hu{\langle\eta_n\rangle}$, while (\ref{?}) and (\ref{??}) show that $\hu{\langle\eta_n\rangle}\cong \hu{I_{G_n}}\cong \hz{\mathbb{Z}_p}$. It is clear that these isomorphisms are not only $G_{n,m}$-linear, but also $G_n$-linear: therefore $\hu{Cyc_n}$ has trivial $G_n$-action.
\begin{cor}\label{Renefruit} With the same hypothesis as in the theorem,
$$
\mathrm{Ker}(\alpha_{[0,1]})=\mathrm{Ker}(\alpha_{[1,0]})=\Pi_n\;.
$$
\end{cor}
\begin{proof}
While proving the theorem we found $\mathrm{Ker}(\alpha_{[1,0]})=\Pi_n$, and we now focus on $\mathrm{Ker}(\alpha_{[0,1]})$. We already know it is cyclic: looking again at (\ref{ker}) we find $\mathrm{Ker}(\alpha_{[0,1]})\subseteq\mathrm{Im}\big(\hu{Cyc_n}\big)\subseteq \hu{\mathcal{O}_n^\times}^{G_n}$: since the isomorphisms $\hz{Q_n}\cong A_n$ are $G_n$-linear, we get $\mathrm{Ker}(\alpha_{[0,1]})\subseteq A_n^{G_n}$. As in the proof of Corollary (\ref{ordine primi}), the assumption $\lambda=0$ is equivalent to $\Pi_n=X^\Gamma$ and $X^\Gamma=A_n^{G_n}$ if $n$ is big enough; putting all together, we have $\mathrm{Ker}(\alpha_{[0,1]})\subseteq \Pi_n$. Since they have the same order thanks to Corollary \ref{ordine primi} together with Theorem \ref{LHF}, the inclusion turns into an equality.
\end{proof}
\rmk As the above Corollary shows, there are indeed two maps $\alpha_{[0,1]}$ and $\alpha_{[1,0]}$ sitting in an exact sequence
$$
0\longrightarrow \Pi_n \longrightarrow A_n\stackrel{\alpha}{\rightarrow}B_n\longrightarrow \mathrm{\B_0}/\eta_0\longrightarrow 0\;,
$$
where $\alpha$ can be either of them. This is the same as in the non-split case, where both $\alpha_{[0,1]}$ and $\alpha_{[1,0]}$ give an isomorphism $A_n\cong B_n$ for $n\gg 0$ if $\lambda=0$ (see \cite{KraSch95}).
\bibliographystyle{amsalpha}
\bibliography{Bibliografia}

\def\cprime{$'$}
\providecommand{\bysame}{\leavevmode\hbox to3em{\hrulefill}\thinspace}
\providecommand{\MR}{\relax\ifhmode\unskip\space\fi MR }
\providecommand{\MRhref}[2]{%
  \href{http://www.ams.org/mathscinet-getitem?mr=#1}{#2}
}
\providecommand{\href}[2]{#2}
\begin{thebibliography}{BNQD01}

\bibitem[BNQD01]{BelNQD01}
J.-R. Belliard and T.~Nguyen Quang~Do, \emph{Formules de classes pour les corps
  ab\'eliens r\'eels}, Ann. Inst. Fourier (Grenoble) \textbf{51} (2001), no.~4,
  903--937.

\bibitem[CF86]{CasFro86}
J.~W.~S. Cassels and A.~Fr{\"o}hlich (eds.), \emph{Algebraic number theory},
  London, Academic Press Inc. [Harcourt Brace Jovanovich Publishers], 1986,
  Reprint of the 1967 original.

\bibitem[FW79]{FerWas79}
Bruce Ferrero and Lawrence~C. Washington, \emph{The {I}wasawa invariant {$\mu
  \sb{p}$} vanishes for abelian number fields}, Ann. of Math. (2) \textbf{109}
  (1979), no.~2, 377--395.

\bibitem[Gil79a]{Gil79-2}
Roland Gillard, \emph{Remarques sur les unit\'es cyclotomiques et les unit\'es
  elliptiques}, J. Number Theory \textbf{11} (1979), no.~1, 21--48.

\bibitem[Gil79b]{Gil79-1}
\bysame, \emph{Unit\'es cyclotomiques, unit\'es semi-locales et {${\bf
  Z}\sb{l}$}-extensions.{II}}, Ann. Inst. Fourier (Grenoble) \textbf{29}
  (1979), no.~4, viii, 1--15.

\bibitem[Gre76]{Gre76}
Ralph Greenberg, \emph{On the {I}wasawa invariants of totally real number
  fields}, Amer. J. Math. \textbf{98} (1976), no.~1, 263--284.

\bibitem[Gre77]{Gre77}
\bysame, \emph{On {$p$}-adic {$L$}-functions and cyclotomic fields. {II}},
  Nagoya Math. J. \textbf{67} (1977), 139--158.

\bibitem[Iwa60]{Iwa60}
Kenkichi Iwasawa, \emph{On local cyclotomic fields}, J. Math. Soc. Japan
  \textbf{12} (1960), 16--21.

\bibitem[Iwa73]{Iwa73}
\bysame, \emph{On {${\bf Z}\sb{l}$}-extensions of algebraic number fields},
  Ann. of Math. (2) \textbf{98} (1973), 246--326.

\bibitem[KS95]{KraSch95}
James~S. Kraft and Ren{\'e} Schoof, \emph{Computing {I}wasawa modules of real
  quadratic number fields}, Compositio Math. \textbf{97} (1995), no.~1-2,
  135--155, Special issue in honour of Frans Oort.

\bibitem[Kuz96]{Kuz96}
L.~V. Kuz{\cprime}min, \emph{On formulas for the class number of real abelian
  fields}, Izv. Ross. Akad. Nauk Ser. Mat. \textbf{60} (1996), no.~4, 43--110.

\bibitem[Lan90]{Lan90}
Serge Lang, \emph{Cyclotomic fields {I} and {II}}, second ed., Graduate Texts
  in Mathematics, vol. 121, Springer-Verlag, New York, 1990, With an appendix
  by Karl Rubin.

\bibitem[NSW00]{NeuSchWin00}
J{\"u}rgen Neukirch, Alexander Schmidt, and Kay Wingberg, \emph{Cohomology of
  number fields}, Grundlehren der Mathematischen Wissenschaften [Fundamental
  Principles of Mathematical Sciences], vol. 323, Springer-Verlag, Berlin,
  2000.

\bibitem[Oza97]{Oza97}
Manabu Ozaki, \emph{On the cyclotomic unit group and the ideal class group of a
  real abelian number field. {I}, {II}}, J. Number Theory \textbf{64} (1997),
  no.~2, 211--222, 223--232.

\bibitem[Rub00]{Rub00}
Karl Rubin, \emph{Euler systems}, Annals of Mathematics Studies, vol. 147,
  Princeton University Press, Princeton, NJ, 2000, Hermann Weyl Lectures. The
  Institute for Advanced Study.

\bibitem[Sch90]{Sch90}
Ren{\'e} Schoof, \emph{The structure of the minus class groups of abelian
  number fields}, S\'eminaire de Th\'eorie des Nombres, Paris 1988--1989,
  Progr. Math., vol.~91, Birkh\"auser Boston, Boston, MA, 1990, pp.~185--204.

\bibitem[Sin81]{Sin80/81}
W.~Sinnott, \emph{On the {S}tickelberger ideal and the circular units of an
  abelian field}, Invent. Math. \textbf{62} (1980/81), no.~2, 181--234.

\bibitem[Was97]{Was97}
Lawrence~C. Washington, \emph{Introduction to cyclotomic fields}, second ed.,
  Graduate Texts in Mathematics, vol.~83, Springer-Verlag, New York, 1997.

\end{thebibliography}
\end{document}